\documentclass[english]{article}
\usepackage[T1]{fontenc}
\usepackage[latin9]{inputenc}
\usepackage{amsthm}
\usepackage{amsmath}
\usepackage{amssymb}

\makeatletter
\theoremstyle{plain}
\newtheorem{thm}{Theorem}[section]
  \theoremstyle{definition}
  \newtheorem{defn}[thm]{Definition}
  \theoremstyle{plain}
  \newtheorem{prop}[thm]{Proposition}
  \theoremstyle{remark}
  \newtheorem{rem}[thm]{Remark}
 \theoremstyle{definition}
  \newtheorem{example}[thm]{Example}
  \theoremstyle{plain}
  \newtheorem{lem}[thm]{Lemma}

\makeatother

\usepackage{babel}

\begin{document}

\title{Some results on condition numbers in convex multiobjective optimization}

\author{M. Bianchi%
\thanks{Dipartimento di Discipline Matematiche, Finanza Matematica ed Econometria,
Universit{à} Cattolica del Sacro Cuore di Milano, Via Necchi 9,
20123 Milano, Italy.%
} \and E. Miglierina%
\thanks{Dipartimento di Discipline Matematiche, Finanza Matematica ed Econometria,
Universit{à} Cattolica del Sacro Cuore di Milano, Via Necchi 9,
20123 Milano, Italy. %
} \and E. Molho%
\thanks{Dipartimento di Economia Politica e Metodi Quantitativi, Università
degli Studi di Pavia, via S. Felice 5, 27100 Pavia, Italy. %
} \and R.Pini%
\thanks{Dipartimento di Statistica, Universit{à} degli Studi di Milano Bicocca,
Via Bicocca degli Arcimboldi 8, 20126 Milano, Italy. %
} }

\date{~}
\maketitle
\begin{abstract}
\noindent Various notions of condition numbers are used to study some
sensitivity aspects of scalar optimization problems. The aim of this
paper is to introduce a notion of condition number to study the case
of a multiobjective optimization problem defined via $m$ convex $C^{1,1}$
objective functions on a given closed ball in $\mathbb{R}^{n}.$ Two
approaches are proposed: the first one adopts a local point of view
around a given solution point, whereas the second one considers the
solution set as a whole. A comparison between the two notions of well-conditioned
problem is developed. We underline that both the condition numbers
introduced in the present work reduce to the same of condition number
proposed by Zolezzi in 2003, in the special case of the scalar optimization
problem considered there. A pseudodistance between functions is defined
such that the condition number provides an upper bound on how far
from a well--conditioned function $f$ a perturbed function $g$ can
be chosen in order that $g$ is well--conditioned too. For both the
local and the global approach an extension of classical Eckart--Young
distance theorem is proved, even if only a special class of perturbations
is considered. \end{abstract}
\begin{quote}
\textbf{\small Key words.}{\small {} Condition number; Eckart--Young
theorem; sensitivity analysis; multiobjective optimization.}{\small \par}

\textbf{\small Mathematics subject classification.}{\small {} 49K40,
90C31, 90C29 }{\small \par}
\end{quote}
\textbf{\small Corresponding Author:}{\small \par}

{\small Elena Molho}{\small \par}

{\small Dipartimento di Economia Politica e Metodi Quantitativi, Università
degli Studi di Pavia, via S. Felice 5, 27100 Pavia, Italy.}{\small \par}

tel. +39-0382986233 - fax +39-0382304226

{\small E-mail: }\texttt{\small molhoe@eco.unipv.it}{\small \par}

\pagebreak{}

\section{Introduction }

The goal of the analysis in parametric optimization is to study how
a change in the data of the problem, represented by a vector of parameters,
affects the solution of the given problem. A first distinction is
usually made between a qualitative and a quantitative approach to
the post--optimal analysis: a qualitative approach concerns the continuity
properties of the optimal solution map, while a quantitative approach
usually involves the evaluation of some kind of derivative of the
optimal solution map. The Lipschitzian behaviour of such a map is
usually considered as a part of the sensitivity analysis and it is
deeply related to the condition number theory. Besides evaluating
bounds on the change in the solutions due to perturbations, condition
numbers provide an estimation on how large a perturbation can be without
disrupting the regular behaviour of the solution map.

The first attempts to consider the numerical implications of Lipschitzian
stability properties and of their reformulations in terms of metric
regularity were developed in the 70s in the pioneering works by Robinson
on generalized equations (see e.g. \cite{Robinson80} and the references
therein). Later Renegar related the distance from ill--posedness to
a notion of condition number for linear programming, thus generalizing
the well--known Eckart--Young theorem of numerical linear algebra.
The condition number introduced in \cite{Renegar1,Renegar2} proved
also to be a fundamental tool in the analysis of the rate of convergence
of interior point methods in linear programming. An extension of the
distance theorem to convex processes was proved in \cite{Lewis1999}.

A few years ago, Zolezzi studied condition numbers in the setting
of quadratic optimization \cite{zolezzi02} and subsequently, more
in general, in problems where a differentiable objective function
with a Lipschitz continuous gradient is minimized on a ball of a Banach
space. The appropriate notion of condition number depends on the class
of perturbations of the given problem which preserve solvability.
In \cite{Zol03} perturbations by linear continuous functions (the
so--called tilt perturbations) are considered, where each perturbed
problem has exactly one solution.

Even if condition numbers are widely used both as tools for sensitivity
analysis and as instruments that may give precious insights on the
numerical aspects of optimization, to our knowledge conditioning techniques
have not yet been developed in the field of vector optimization.

The aim of the present work is a first attempt to define an appropriate
notion of condition number for multiobjective optimization problems.
We limit our study to the case of $m$ convex differentiable objective
functions with Lipschitz continuous gradients and we consider the
weakly Pareto efficient solution map on a ball in $\mathbb{R}^{n}$
under componentwise uniform tilt perturbations, in the sense that
each component of the objective function will receive the same tilt
perturbation. We avoid any requirement of uniqueness of the solution,
since it is unduly restrictive in the setting of vector optimization.
We will consider two different approaches. The first one, the pointwise
approach, will focus on a given weakly efficient solution $\bar{x}.$
The pointwise condition number is defined as the Lipschitz modulus
of the weakly efficient solution map at $\bar{x}$ for $p=0,$ where
the value $p=0$ characterizes the unperturbed problem.

On the other hand, one can follow a global approach where the whole
weakly efficient solution set on a given ball is considered under
tilt perturbations. The global condition number is built as a direct
sensitivity measure on the weakly efficient solution map of the perturbed
problems. Moreover we show that the local and the global definitions
of condition numbers introduced here allow us to build a consistent
framework. Indeed, the global condition number for a given multiobjective
optimization problem is finite if and only if the local one is finite
at every weakly efficient solution of the same problem.

In both cases a distance theorem is proved. A pseudodistance between
functions is defined such that the condition number provides an upper
bound on how far from a well--conditioned function $f$ a perturbed
function $g$ can be chosen in order that $g$ is well--conditioned
too.

The two notions of condition number considered in the present work,
either following the pointwise approach or the global one, coincide
with the condition number $c_{2}$ proposed in \cite{Zol03} in the
special scalar optimization problem discussed there. They can both
be considered as tools to extend the classical Eckart--Young distance
theorem to the setting of multiobjective optimization, even if only
a special class of perturbations is considered.

It is interesting to remark that in both approaches strong convexity
plays the role of a sufficient condition to obtain well--conditioning.

The paper is organized in five section. Section 2 contains some notations
and preliminaries. In Sections 3 and 4 our main results are stated
and proved. Indeed, in Section 3 we introduce a condition number for
a multiobjective optimization problem, following a local point of
view, whereas Section 4 is devoted to the study of a notion of condition
number adopting a global point of view. Moreover, in both these sections,
we prove an Eckart--Young type theorem related, respectively, to the
local and the global condition number.

\section{Notations and preliminaries}

In order to study the behaviour of the set--valued solution maps of
a multiobjective optimization problem, we need to consider some classical
notions of continuity and Lipschitz continuity for set--valued maps.
For a detailed exposition see, e.g., \cite{DonRock09}, or \cite{BM06}.

A set--valued map $T:A\subset\mathbb{R}^{n}\rightrightarrows\mathbb{R}^{m}$
is \textit{upper semicontinuous} at $x\in A$ if for every open set
$\mathcal{V}(T(x))$ such that $T(x)\subseteq\mathcal{V}(T(x))$ there
exists a neighborhood $\mathcal{U}(x)$ of $x$ such that for every
$x'\in\mathcal{U}(x),$ $T(x')\subset\mathcal{V}(T(x)).$

If $T:A\subset\mathbb{R}^{n}\rightrightarrows\mathbb{R}^{m}$ is a
closed--valued map with compact values at $x\in A,$ then $T$ is
upper semicontinuous at $x$ if and only if its graph is closed at
$x$, i.e., for each sequence $\{x_{n}\}\subset A$, $x_{n}\to x$,
and each $y_{n}\in T(x_{n})$, $y_{n}\to y,$ then $y\in T(x)$.

Lipschitzian properties of maps play a crucial role in many aspects
of variational analysis. We consider here some well--known generalizations
of the notion of Lipschitz continuity to set--valued maps that will
allow us to build a meaningful notion of condition number for multiobjective
optimization.

First let us recall the notion of Hausdorff distance between two sets.
Let $A,B\subset\mathbb{R}^{n}.$ The Hausdorff distance $d_{H}$ between
$A$ and $B$ is defined by\[
d_{H}(A,B)=\max\left\{ e(A,B),e(B,A)\right\} ,\]
 where $e(A,B)=\sup_{a\in A}d(a,B)=\sup_{a\in A}\inf_{b\in B}\left\Vert a-b\right\Vert $
is the excess functional; in particular, $e(\emptyset,A)=0,\, e(\emptyset,\emptyset)=0$
and $e(A,\emptyset)=+\infty.$

In the sequel, we will denote by $B(x,r)$ the closed ball in $\mathbb{R}^{n}$
centered at $x$ and with radius $r.$
\begin{defn}
A set--valued map $T:\mathbb{R}^{n}\rightrightarrows\mathbb{R}^{m}$
is \emph{Lipschitz continuous relative to a nonempty set $A\subset\mathbb{R}^{n}$}
if $A\subset\mathrm{dom}(T),$ $T$ is closed--valued on $A,$ and
there exists a Lipschitz constant $k\ge0$ such that \begin{equation}
T(x')\subset T(x)+k\|x'-x\|B(0,1)\quad\forall x,x'\in A.\label{eq:lipschitz}\end{equation}
 The infimum of the Lipschitz constants on A is called the \emph{Lipschitz
modulus of $T$ on $A$} and it is denoted by $\mathrm{lip}\left(T;A\right)$.

If $k<1$, the set--valued map $T$ is said to be a \emph{contraction}
of constant $k.$
\end{defn}
An equivalent reformulation of condition \eqref{eq:lipschitz} can
be given in terms of the Hausdorff distance as follows: \[
d_{H}(T(x'),T(x))\le k\|x'-x\|\quad\forall x,x'\in A.\]

A local version of Lipschitz continuity of a map is the so called
\emph{Aubin property}.
\begin{defn}
A set--valued map $T:\mathbb{R}^{n}\rightrightarrows\mathbb{R}^{m}$
is said to have the \emph{Aubin property at $\bar{x}$ for $\bar{y}\in T(\bar{x})$}
if there exist $k\ge0,$ a neighborhood $\mathcal{U}(\bar{x})$ of
$\bar{x}$ and a neighborhood $\mathcal{V}(\bar{y})$ of $\bar{y}$
such that \begin{equation}
T(x')\cap\mathcal{V}(\bar{y})\subset T(x)+k\|x-x'\|B(0,1),\qquad\forall x,x'\in\mathcal{U}(\bar{x}).\label{Aubin}\end{equation}
 The infimum of $k$ such that (\ref{Aubin}) holds is called the
\emph{Lipschitz modulus of $T$ at $\bar{x}$ for $\bar{y}$} and
is denoted by $\mathrm{lip}(T;\bar{x}|\bar{y}).$
\end{defn}
When $T$ is single--valued on $A,$ then the Lipschitz modulus of
$T$ on $A$ corresponds to the usual definition:\[
\mathrm{lip}(T;A)=\underset{x,x'\in A}{\mathrm{sup}}\frac{\left\Vert T(x)-T(x')\right\Vert }{\left\Vert x-x'\right\Vert },\]
 while the Lipschitz modulus of $T$ at $\bar{x}$ for $T(\bar{x})$
equals $\mathrm{lip}(T;\bar{x}),$ i.e., the usual Lipschitz modulus
at $\bar{x}:$ \[
\mathrm{lip}(T;\bar{x})=\limsup_{x,x'\to\bar{x},\, x\neq x'}\frac{\left\Vert T(x)-T(x')\right\Vert }{\left\Vert x-x'\right\Vert }.\]

Another crucial notion concerning set--valued maps is metric regularity,
that is related to the Lipschitzian properties of inverse maps.
\begin{defn}
Let $T:\mathbb{R}^{n}\rightrightarrows\mathbb{R}^{m}$ with $\mathrm{dom}(T)\neq\emptyset$
and let $(\bar{x},\bar{y})\in\mathrm{gph}\, T$. The set--valued map
$T$ is said to be \emph{metrically regular at $\bar{x}$ for $\bar{y}\in T(\bar{x})$}
with modulus $\mu\ge0$ if there exist a neighborhood $\mathcal{U}(\bar{x})$
of $\bar{x}$ and a neighborhood $\mathcal{V}(\bar{y})$ of $\bar{y}$
such that \begin{equation}
d(x,T^{-1}(y))\leq\mu d(y,T(x))\label{eq:reg}\end{equation}
 for all $x\in\mathcal{U}(\bar{x})$ and for all $y\in\mathcal{V}(\bar{y}).$
The infimum of $\mu$ such that (\ref{eq:reg}) holds for some $\mathcal{U}(\bar{x})$
and $\mathcal{V}(\bar{y})$ is called the \emph{regularity modulus
for $T$ at $\bar{x}$ for $\bar{y}\in T(\bar{x}),$} and is denoted
by $\mathrm{reg}(T;\bar{x}|\bar{y}).$
\end{defn}
A fundamental characterization of the metric regularity of a given
map is the following (see, e.g., Theorem 3E.6 in \cite{DonRock09}):
\begin{prop}
\label{pro Aubin reg} Let $(\bar{x},\bar{y})\in\mathrm{gph}\, T,$
where $T:\mathbb{R}^{n}\rightrightarrows\mathbb{R}^{m}.$ Then $T$
satisfies the Aubin property at $\bar{x}$ for $\bar{y}$ if and only
if its inverse $T^{-1}:\mathbb{R}^{m}\rightrightarrows\mathbb{R}^{n}$
is metrically regular at $\bar{y}$ for $\bar{x}$ with the same modulus,
i.e., $\mathrm{lip}(T;\bar{x}|\bar{y})=\mathrm{reg}(T^{-1};\bar{y}|\bar{x}).$
\end{prop}
In order to measure how far a metrically regular map $T$ can be perturbed
by a linear operator without loosing metric regularity, in \cite{DonLewRock2003}
Dontchev, Lewis and Rockafellar introduced the notion of radius of
metric regularity.
\begin{defn}
Let $T:\mathbb{R}^{n}\rightrightarrows\mathbb{R}^{m}$ with $\mathrm{dom}(T)\neq\textrm{Ø}$
and let $(\bar{x},\bar{y})\in\mathrm{gph}\, T$. The \emph{radius
of metric regularity of $T$ at $\bar{x}$ for $\bar{y}\in T(\bar{x})$}
is the value\[
\mathrm{rad}\left(T;\bar{x}|\bar{y}\right)=\underset{G\in\mathcal{L}(\mathbb{R}^{n},\mathbb{R}^{m})}{\mathrm{inf}}\left\{ \left\Vert G\right\Vert :T+G\,\mathrm{is\, not\, metrically\, regular\, at\,}\bar{x}\,\mathrm{for}\,(\bar{y}+G(\bar{x}))\right\} .\]

\end{defn}
In the Euclidean setting the radius $\mathrm{rad}(T;\bar{x}|\bar{y})$
defined above turns out to be meaningful also for a larger class of
perturbation maps, namely \[
\mathrm{rad}\left(T;\bar{x}|\bar{y}\right)=\underset{G\in\mathcal{G}}{\mathrm{min}}\left\{ \mathrm{lip}(G;\bar{x}):\, T+G\,\mathrm{is\, not\, metrically\, regular\, at\,}\bar{x}\,\mathrm{for}\,(\bar{y}+G(\bar{x}))\right\} \]
 where $\mathcal{G}$ is the collection of all the functions $G:\mathbb{R}^{n}\rightarrow\mathbb{R}^{m}$
that are Lipschitz continuous at $\bar{x}$.

The radius of metric regularity can be characterized as the reciprocal
of the regularity modulus. We recall that the local closedness of
a set at a point $x$ means that some ball $B(x,r)$ has closed intersection
with the set.
\begin{prop}
\label{radius}(see Theorem 1.5 in \cite{DonLewRock2003}). Let $T:\mathbb{R}^{n}\rightrightarrows\mathbb{R}^{m}$
with $\mathrm{dom}(T)\neq\textrm{Ø}$ and let $(\bar{x},\bar{y})\in\mathrm{gph}\, T$
be a point such that $\mathrm{gph}\, T$ is locally closed at $(\bar{x},\bar{y}).$
Then \[
\mathrm{rad}\left(T;\bar{x}|\bar{y}\right)=\frac{1}{\mathrm{reg}(T;\bar{x}|\bar{y})}.\]

\end{prop}

\subsection{Parametric multiobjective problems}

Now we introduce a parametric multiobjective optimization problem
that will be considered in the sequel in order to develop a suitable
notion of condition number. In the present work we restrict to the
special case of componentwise uniform tilt perturbations, in the sense
that we perturb all the components of the objective function with
the same linear term. Although some of the results obtained could
be stated in a more general setting, the main results require some
regularity of the data. This explain why we will confine ourselves
to a more restrictive framework.

We denote by $\mathcal{C}^{1,1}(B(0,r))$ the set of all vector--valued
functions $f:B(0,r)\subset\mathbb{R}^{n}\to\mathbb{R}^{m}$ such that
$f_{i},$ $i=1,2,...,m,$ is differentiable at each interior point
of $B(0,r),$ and whose gradient $\nabla f_{i}$ can be extended to
the closed ball $B(0,r)$ in such a way that it is Lipschitz continuous
on the whole set $B(0,r).$ Let $\mathcal{U}(0)$ be a neighborhood
of the origin in $\mathbb{R}^{n},$ and let $p\in\mathcal{U}(0).$
We denote by $f^{p}$ the perturbation of $f$ defined by \[
f^{p}(x)=f(x)-[p]\cdot x,\]
 where $[p]$ denotes a matrix with $m$ rows and $n$ columns such
that all the rows are equal to $p.$

Given a function $f\in\mathcal{C}^{1,1}(B(0,r))$ we consider the
following parametric multiobjective optimization problem

\begin{equation}
\min_{x\in B(0,r)}f^{p}(x),\tag{VO\ensuremath{_{p}}}\label{eq:VOP}\end{equation}
 where the ordering cone in $\mathbb{R}^{m}$ is given by the nonnegative
orthant \[
\mathbb{R}_{+}^{m}=\{x=(x_{1},x_{2},\dots,x_{m})\in\mathbb{R}^{m}:\, x_{i}\ge0,\, i=1,2,\dots,m\}.\]

We will denote by $WE_{f}$ the set--valued map $WE_{f}:\mathcal{U}(0)\subset\mathbb{R}^{n}\rightrightarrows\mathbb{R}^{n}$
of the weakly efficient solutions of problem {\eqref{eq:VOP}},
i.e., $WE_{f}(p)$ is the set of points $x\in B(0,r)$ such that \[
(f^{p}(x)-\mathrm{int}(\mathbb{R}_{+}^{m}))\cap f^{p}(B(0,r))=\emptyset.\]
 Since the set $f^{p}(B(0,r))$ is compact$,$ then the set $WE_{f}(p)\subseteq B(0,r)$
is nonempty and closed, for every $p\in\mathcal{U}(0)$ (see, for
instance, Theorem 6.5 in \cite{Jahn04}). Furthermore, the set--valued
map $WE_{f}$ turns out to be upper semicontinuous at $0$ (see Corollary
4.6, Ch. 4 in \cite{Luc}).

First order optimality conditions play a key role in the development
of a condition number theory. Indeed, in the present work the sensitivity
of the solution map $WE_{f}$ will be investigated by means of the
effect that a perturbation of the objective function produces on the
inclusion that is equivalent to the first order optimality conditions.

To any function $f\in\mathcal{C}^{1,1}(B(0,r))$, we associate the
set--valued map $H_{f}:B(0,r)\rightrightarrows\mathbb{R}^{n}$ defined
as \[
x\mapsto H_{f}(x)=\left\{ \sum_{i=1}^{m}\lambda_{i}\nabla f_{i}(x),\quad\sum_{i=1}^{m}\lambda_{i}=1,\lambda_{i}\ge0\right\} .\]
 It is worthwhile noticing that this map has nonempty, compact and
convex values; moreover, it can be easily proved that it has a closed
graph. Set $s_{f}(x)=d(0,H_{f}(x)).$

In the next proposition the Lipschitz continuity of the maps $H_{f}$
and $s_{f}$ is proved (for more details on $H_{f}(x)$ and $s_{f}(x)$
we refer to \cite{Mig04}).
\begin{prop}
\label{Lip H} Let $f\in\mathcal{C}^{1,1}(B(0,r))$. Then, the set--valued
map $H_{f}$ and the function $s_{f}$ are Lipschitz continuous on
$B(0,r)$ with the same Lipschitz constant given by $K=\max_{i=1,\dots,m}\mathrm{lip}(\nabla f_{i};B(0,r)).$ \end{prop}
\begin{proof}
Let $x,y\in B(0,r);$ then \[
\nabla f_{i}(x)-\nabla f_{i}(y)\in\mathrm{lip}(\nabla f_{i};B(0,r))\, B(0,\|x-y\|)\subseteq K\, B(0,\|x-y\|).\]
 By considering the convex hull of the first term, we obtain\[
\sum_{i=1}^{n}\lambda_{i}(\nabla f_{i}(x)-\nabla f_{i}(y))\in K\, B(0,\|x-y\|),\]
 whenever $\sum_{i=1}^{m}\lambda_{i}=1,$ $\lambda_{i}\ge0.$ This
implies that \[
H_{f}(x)\subseteq H_{f}(y)+K\, B(0,\|x-y\|),\]
 i.e., $H_{f}$ is Lipschitz continuous on $B(0,r)$ with Lipschitz
constant $K$. This yields that $s_{f}:B(0,r)\to\mathbb{R}$ is Lipschitz
continuous on $B(0,r)$ with the same Lipschitz constant $K$ (see
\cite{RockWets98}, p. 368).
\end{proof}
The above--mentioned map $H_{f}$ is an essential tool to deal with
useful first order conditions. As a matter of fact it is well--known
that any interior solution $\bar{x}\in WE_{f}(p)\cap\mathrm{int}(B(0,r))$
of problem \eqref{eq:VOP} satisfies the inclusion $0\in H_{f^{p}}(\bar{x})$
or, equivalently, $s_{f^{p}}(\bar{x})=0$. It is easy to see that
\[
H_{f^{p}}(x)=H_{f}(x)-\left\{ p\right\} ,\]
 hence $0\in H_{f^{p}}(x)$ if and only if $p\in H_{f}(x)$. Therefore,
the following inclusion holds:\[
WE_{f}(p)\cap\mathrm{int}(B(0,r))\subset H_{f}^{-1}(p).\]

If $f$ is an $\mathbb{R}_{+}^{m}$--convex function on $B(0,r)$,
i.e., all its components $f_{i},$ $i=1,2,...,m,$ are convex functions
on $B(0,r),$ and $\overline{x}\in\mathrm{int}(B(0,r))$, then $s_{f_{p}}(\overline{x})=0,$
and hence $0\in H_{f^{p}}$ if and only if $\overline{x}\in WE_{f}(p)$.
Therefore $\overline{x}\in WE_{f}(p)\cap\mathrm{int}(B(0,L))$ if
and only if $\overline{x}\in H_{f}^{-1}(p).$ Moreover, if $WE_{f}(p)\subset\mathrm{int}(B(0,r)),$
the following characterization of the weakly efficient solution map
holds: \begin{equation}
WE_{f}(p)=H_{f}^{-1}(p).\label{eq:4}\end{equation}

\subsection{Condition number for scalar optimization problems}

For the convenience of the reader we repeat some relevant material
from \cite{Zol03} without proofs, thus making our exposition self--contained.
Let $E$ be a real Banach space and $E^{*}$ its dual, and let $\langle\cdot,\cdot\rangle$
denote the duality pairing. $\mathcal{C}^{1,1}(B(0,L))$ denotes the
class of Fr{é}chet differentiable functions on the closed ball $B(0,L)$
such that the Fr{é}chet derivative $Df$ is a Lipschitz function
on $B(0,L).$ Consider the problem \[
\min_{B(0,L)}f_{p}=\min_{B(0,L)}(f-\langle p,\cdot\rangle).\]
 Under the assumption that the solution is unique, for every small
$p,$ we can set \[
m:\mathcal{U}(0)\subset E^{*}\to E,\quad m(p)=\mathrm{argmin}(B(0,L),f_{p}).\]
 The condition number $\mathrm{cond}(f)$ is the extended real number
defined as follows: \begin{equation}
\mathrm{cond}(f)=\limsup_{p,q\to0,\, p\neq q}\frac{\|m(p)-m(q)\|}{\|p-q\|}\label{c2(I)}\end{equation}

If $m(0)=\bar{x},$ by Proposition \ref{pro Aubin reg} the condition
number $\mathrm{cond}(f)$ agrees with the regularity modulus of the
set--valued map $m_{f}^{-1}$ at $(\bar{x},0),$ i.e., \begin{equation}
\mathrm{cond}(f)=\mathrm{reg}(m^{-1};\bar{x}|0).\label{c2(II)}\end{equation}

The class $\mathcal{C}^{1,1}(B(0,r))$ can be endowed with the pseudodistance
\begin{equation}
d_{Z}(f_{1},f_{2})=\sup\left\{ \frac{\|Df_{1}(x)-Df_{2}(x)-Df_{1}(x')+Df_{2}(x')\|}{\|x-x'\|}\right\} \label{distanza zolezzi}\end{equation}
 where $x,x'\in B(0,L),$ $x\neq x'.$

Denote by $T_{1}$ the class of functions $f\in\mathcal{C}^{1,1}(B(0,L))$
such that
\begin{itemize}
\item $\mathrm{argmin}(B(0,L),f_{p})\neq\emptyset$ for small $p;$
\item $\mathrm{argmin}(B(0,L),f)=\{0\};$
\item $p\mapsto\mathrm{argmin}(B(0,L),f_{p})$ is upper semicontinuous at
$p=0.$
\end{itemize}
The next class $W_{1}$ can be thought of as the class of \lq\lq
good\textquotedbl{} functions, giving rise to well--conditioned problems:
$f\in W_{1}$ if
\begin{itemize}
\item $\mathrm{argmin}(B(0,L),f_{p})$ is a singleton for $p$ small;
\item $\mathrm{cond}(f)<+\infty.$
\end{itemize}
The ill--conditioned functions $I_{1}=\{g\in T_{1}:g\notin W_{1}\}$
satisfy the following result:
\begin{thm}
(see \cite{Zol03}, Theorem 3.1). Let $f\in T_{1}\cap W_{1}$ with
$Df$ one--to--one near $0.$ Then \[
d_{Z}(f,I_{1})\ge\frac{1}{\mathrm{cond}(f)}.\]

\end{thm}

\section{Condition number: a pointwise definition }

The definition of condition number (\ref{c2(I)}) in the scalar case,
and its equivalent formulation (\ref{c2(II)}) give rise to different
approaches in the setting of multiobjective optimization. This section
is devoted to the analysis of the pointwise conditioning, that will
extend the scalar approach summarized in the formula for condition
number given in (\ref{c2(II)}). We will focus on a fixed efficient
solution of the multiobjective optimization problem (VO$_{0}$).
\begin{defn}
\label{def: cn puntuale} Let $\bar{x}\in WE_{f}(0).$ The \emph{condition
number of $f$ at $\bar{x}$} is the extended real number \[
c(\bar{x},f)=\mathrm{reg}(WE_{f}^{-1};\bar{x}|0).\]

\end{defn}
By Proposition \ref{pro Aubin reg} the regularity modulus of a set--valued
map coincides with the Lipschitz modulus of its inverse. Hence we
deduce that\[
c(\bar{x},f)=\mathrm{lip}(WE_{f};0|\bar{x}).\]

In force of Proposition \ref{radius} we can use the condition number
$c(\bar{x},f)$ to characterize the radius of metric regularity of
the map $WE_{f}^{-1}$ at $\bar{x}$ for $0$ as follows:\[
\mathrm{rad}(WE_{f}^{-1};\bar{x}|0)=\frac{1}{c(\bar{x},f)}.\]

This result, which holds without any assumption of smoothness on $f,$
allows us to establish a first version of distance theorem: the condition
number bounds the linear continuous perturbations that can be applied
to the inverse of the solution map $WE_{f}$ without loosing metric
regularity. As already remarked in \cite{Zol03}, the distance to
ill--conditioning defined through $\mathrm{rad}(WE_{f}^{-1};\bar{x}|0)$
cannot be considered as a variational notion since the addition of
a linear perturbation to $WE_{f}^{-1}$ does not correspond to an
additive perturbation of the original objective function.

In order to define well--conditioned problems and to use the pointwise
condition number $c(\bar{x},f)$ to provide a lower bound of the distance
from ill--conditioning of a given well--conditioned multiobjective
optimization problem, it is necessary to put some restrictions on
$f.$ Indeed, one of the main difficulties to face when dealing with
vector optimization is to provide conditions that fully characterize
the weakly efficient solutions of the optimization problem. In the
scalar case, Zolezzi considered essentially functions with a unique
global minimum point at $0,$ and such that $Df$ is one--to--one
near $0,$ thereby ensuring that the zeroes of the gradient completely
identifies the solutions of the problem. In this framework, the necessary
first order conditions for extrema play a crucial role. Clearly, strictly
convex functions with a minimum point at $0$ would be suitable.

In the vector--valued case we will restrict our analysis to the class
of $\mathbb{R}_{+}^{m}$--convex functions; the advantage of using
functions of this class lies in the fact that the first order conditions
completely characterize the weakly efficient solutions. In addition,
throughout the sequel of the paper we will assume that \[
WE_{f}(0)\subset\mathrm{int}(B(0,r)).\]
 We underline that, since the set--valued map $WE_{f}$ is upper semicontinuous,
the inclusion above implies that there exists a positive real number
$\delta_{f}$ such that \begin{equation}
WE_{f}(p)\subset{\rm int}(B(0,r))\quad{\rm for}\,{\rm every}\, p\in\mathbb{R}^{n},\,\left\Vert p\right\Vert <\delta_{f}.\label{eq:WE dentro int}\end{equation}

Let us given the following
\begin{defn}
Let $\bar{x}$ be a point in $WE_{f}(0).$ A function $f\in\mathcal{C}^{1,1}(B(0,r))$
belongs to the class $T_{1}(\bar{x})$ if the following conditions
hold:
\begin{itemize}
\item $f$ is $\mathbb{R}_{+}^{m}$--convex;
\item $c(\bar{x},f)>0.$
\end{itemize}
\end{defn}
A condition that entails the positivity of the condition number introduced
in Definition \ref{def: cn puntuale} can be established. The following
proposition extends a result proved in Lemma 3.2 in \cite{Zol03}.
\begin{prop}
\label{pro:positivo} \label{pr:1} Let $\bar{x}\in WE_{f}(0).$ If
there exist $\kappa>0$ and $\{p_{s}\}\in\mathbb{R}^{n}\setminus\{0\},$
$p_{s}\to0,$ such that \begin{equation}
\|p_{s}\|\le\kappa\, d(\bar{x},WE_{f}(p_{s})),\label{ipotesi2}\end{equation}
 then $c(\bar{x},f)>0$. \end{prop}
\begin{proof}
From the definition of $c(\bar{x},f),$ for any $\epsilon>0$ there
exist $\mathcal{U}(\overline{x})$ and $\mathcal{V}(0)$ such that
\[
d(x,WE_{f}(p))\le(c(\bar{x},f)+\epsilon)\, d(p,WE_{f}^{-1}(x))\]
 for every $x\in\mathcal{U}(\overline{x})$ and $p\in\mathcal{V}(0).$
Now let us choose $x=\bar{x}$ and $p=p_{s};$ we obtain, for $s$
big enough, \[
d(\bar{x},WE_{f}(p_{s}))\le(c(\bar{x},f)+\epsilon)\,\|p_{s}\|\le(c(\bar{x},f)+\epsilon)\,\kappa\, d(\bar{x},WE_{f}(p_{s})).\]
 From (\ref{ipotesi2}), $d(\bar{x},WE_{f}(p_{s}))\neq0$ and the
inequality above implies \[
(c(\bar{x},f)+\epsilon)\kappa\ge1.\]
 Since this holds for any $\epsilon>0,$ we conclude that $c(\bar{x},f)>0.$ \end{proof}
\begin{rem}
Condition (\ref{ipotesi2}) weakens the assumption of Lipschitz continuity
of the gradient $Df$ at $\bar{x}=0$ required by Zolezzi in \cite{Zol03}.
Indeed, in the scalar case, $Df(0)=0$ since $x=0$ is a global minimizer
of $f$ in $\mathrm{int}(B(0,r)),$ and, thanks to the assumption
of upper semicontinuity of the map $m=WE_{f},$ $WE_{f}(p)$ is also
an internal minimizer of $f_{p}$ for $p$ small enough, implying
that $Df(WE_{f}(p))=p.$ Thus the Lipschitz assumption on the gradient
gives: \[
\|p\|=\|Df(WE_{f}(p))-Df(WE_{f}(0))\|\le\kappa\|WE_{f}(p)\|=\kappa\, d(0,WE_{f}(p)).\]

\end{rem}
Now we introduce a class of functions that will give rise to well--conditioned
problems at $\bar{x}.$
\begin{defn}
If $\bar{x}\in WE_{f}(0),$ a function $f\in\mathcal{C}^{1,1}(B(0,r))$
is said to belong to the class $W_{1}(\bar{x})$ if $c(\bar{x},f)<+\infty.$
\end{defn}
A stronger convexity assumption on the objective function will entail
a finite condition number. We recall that a function $f:B(0,r)\subset\mathbb{R}^{n}\to\mathbb{R}$
is strongly convex if there exists $\alpha>0$ such that \[
f((1-t)x+tx')\le(1-t)f(x)+tf(x')-\alpha(1-t)t\|x-x'\|^{2},\quad\forall x,x'\in B(0,r),\quad\forall t\in[0,1]\]
 (see, e.g., \cite{vial83}). If $f$ is differentiable, strong convexity
is equivalent to the strong monotonicity of $\nabla f$ on $B(0,r),$
i.e., \[
<\nabla f(x)-\nabla f(x'),x-x'>\ge2\alpha\|x-x'\|^{2},\quad\forall x,x'\in B(0,r).\]
 The following proposition holds:
\begin{prop}
\label{c finito} Let $f\in\mathcal{C}^{1,1}(B(0,r))$. If $f_{i}$
is strongly convex, for every $i=1,2,\dots,m,$ then $c(\bar{x},f)<+\infty$,
for every $\bar{x}\in WE_{f}(0).$ \end{prop}
\begin{proof}
From the assumptions, the function $\nabla f_{i}$ is strongly monotone
for every $i,$ i.e., there exists $k_{i}>0$ such that \[
\langle\nabla f_{i}(x')-\nabla f_{i}(x),x'-x\rangle\ge k_{i}\|x'-x\|^{2},\quad\forall x',x\in B(0,r),\; i=1,2,\dots,m.\]
 It follows easily that the assumptions of Theorem 5.2 in \cite{LeeKimLeeYen98}
are fulfilled if we take $F_{i}(x,p)=\nabla f_{i}(x)-p,$ thereby
implying that the map $p\mapsto WE_{f}(p)$ is Lipschitz on $B(0,r).$
\end{proof}
It is interesting to notice that the convexity assumptions in the
above proposition cannot be weakened. Indeed, even the strict convexity
of the functions $f_{i}$ is not enough to ensure the finiteness of
${c}(\bar{x},f)$ at a point $\bar{x}\in WE_{f}(0).$
\begin{example}
Let us consider the strictly convex function $f:[-1,1]\to\mathbb{R}^{2}$
defined by $f(x)=(x^{2},x^{4}).$ The unperturbed problem has a unique
weakly efficient solution $\bar{x}=0.$ Let us consider the perturbed
problem \eqref{eq:VOP}, where $p\in\mathbb{R}$ and the objective
function is $f^{p}(x)=(x^{2}-px,x^{4}-px)$. The weakly efficient
solution set $WE_{f}(p)$ is the closed segment $I{(p)}=\left\{ x\in\left[-1,1\right]:p/2\leq x\leq\sqrt[3]{p/4}\right\} $.
Suppose, by contradiction, that $c(0,f)=\mathrm{lip}(WE_{f};0|0)<+\infty;$
this means that there exists $k>0,$ $\mathcal{U}(0)$ and $\mathcal{V}(0)$
such that \[
I(p)\cap\mathcal{V}(0)\subset I(q)+k|p-q|B(0,1),\qquad\forall p,q\in\mathcal{U}(0).\]
 Let $q=0;$ then \[
I(p)\cap\mathcal{V}(0)\subset k|p|B(0,1),\qquad\forall p\in\mathcal{U}(0),\]
 or, equivalently, for $|p|$ small enough, \[
|\sqrt[3]{p/4}|\le k|p|,\qquad\forall p\in\mathcal{U}(0),\]
 which is false.
\end{example}
In the development of a condition number theory for multiobjective
optimization a fundamental issue is the possibility to use the condition
number to establish a lower bound on the distance to ill--conditioning.
How far can we move from a well--conditioned vector function $f\in{W}_{1}(\bar{x})\cap T_{1}(\bar{x})$
to a function $g$ that is {}``close'' according to a suitable distance,
without leaving the class $W_{1}(\bar{x})?$

In order to tackle this problem, we introduce a pseudodistance $d^{*}$
in the space of vector--valued functions $\mathcal{C}^{1,1}(B(0,r))$
as follows: \begin{equation}
d^{*}(f,g)=\underset{\lambda_{i}\ge0,\,\sum_{i}\lambda_{i}=1}{\max}\, d_{Z}\left({\sum_{i}}\lambda_{i}f_{i},{\sum}_{i}\lambda_{i}g_{i}\right),\label{d*}\end{equation}
 where $d_{Z}$ denotes the pseudodistance defined in (\ref{distanza zolezzi})
for the scalar functions.

In the sequel, we will consider the particular class $P_{f}$ of functions
obtained by perturbing each component of $f$ by the same real--valued
function $h:B(0,r)\to\mathbb{R},\, h\in\mathcal{C}^{1,1}(B(0,r))$,
i.e., \begin{equation}
P_{f}=\{g=f+h\mathbf{e},\; h:B(0,r)\to\mathbb{R},\, h\in\mathcal{C}^{1,1}(B(0,r)),\,{\mathbf{e}}=(1,1,\dots,1)\}.\label{Pf}\end{equation}
 If $g\in P_{f},$ the distance $d^{*}(f,g)$ turns out to be the
following: \begin{equation}
d^{*}(f,g)=\mathrm{lip}(\nabla h;B(0,r)).\label{d*(f,g)}\end{equation}
 The next proposition allows us to establish an upper bound on the
regularity modulus of the perturbed function $g.$
\begin{prop}
\label{prop th. 3F.1} Let $f\in T_{1}(\bar{x})\cap W_{1}(\bar{x}),$
where $\bar{x}\in WE_{f}(0),$ and let $g$ be a function in $P_{f},$
$g=f+h\mathbf{e}$ with $h$ such that $\mathrm{lip}(\nabla h;\bar{x})\cdot c(\bar{x},f)<1.$
Then \[
\mathrm{reg}(H_{g};\bar{x}|\nabla h(\bar{x}))\le\frac{c(\bar{x},f)}{1-c(\bar{x},f)\cdot\mathrm{lip}(\nabla h;\bar{x})}.\]
\end{prop}
\begin{proof}
Let us consider the set--valued map $H_{f}:B(0,r)\rightrightarrows\mathbb{R}^{n}$
and the function $\nabla h:B(0,r)\rightarrow\mathbb{R}^{n}.$ Since
$\bar{x}\in WE_{f}(0),$ we have that $(\bar{x},0)\in\mathrm{gph}\,(H_{f}),$
where $\mathrm{gph}\,(H_{f})$ is closed. In addition, by the $\mathbb{R}_{+}^{m}$--convexity
of $f,$ it holds \[
c(\bar{x},f)=\mathrm{reg}(WE_{f}^{-1};\bar{x}|0)=\mathrm{reg}(H_{f};\bar{x}|0).\]
 The equality $H_{g}(x)=H_{f}(x)+\nabla h(x)$ holds for every $x\in B(0,r).$
Therefore, by Theorem 3F.1 in \cite{DonRock09}, the assertion is
proved.
\end{proof}
Now we can reformulate the former result as a distance theorem for
the proposed notion of pointwise condition number.
\begin{thm}
\label{cor distance 1} Let $f\in T_{1}(\bar{x})\cap W_{1}(\bar{x}),$
where $\bar{x}\in WE_{f}(0),$ and $g\in P_{f}\cap T_{1}(\bar{x})$
be such that \[
d^{*}(g,f)<\frac{1}{c(\bar{x},f)}.\]
 If $\nabla h(\bar{x})=0,$ then $g\in W_{1}(\bar{x}).$ \end{thm}
\begin{proof}
We see at once that $\mathrm{lip}(\nabla h;x)\le\mathrm{lip}(\nabla h;B(0,r)),$
for every $x\in B(0,r).$ Therefore, by Proposition \ref{prop th. 3F.1}
we get \[
\mathrm{reg}(H_{g};\bar{x}|0)\le\frac{c(\bar{x},f)}{1-c(\bar{x},f)\cdot\mathrm{lip}(\nabla h;\bar{x}))}\leq\frac{c(\bar{x},f)}{1-c(\bar{x},f)\cdot\mathrm{lip}(\nabla h;B(0,r))}.\]
 Since, by assumptions, $g\in T_{1}(\bar{x})$, the condition $0\in H_{g}(\bar{x})$
is equivalent to $\bar{x}\in WE_{g}(0);$ therefore, $\mathrm{reg}(H_{g};\bar{x}|0)=c(\bar{x},g).$
From (\ref{d*(f,g)}), the conclusion follows.
\end{proof}

\section{Condition number: a global definition}

\noindent In the present section we introduce a global definition
of condition number that extends the notion given in (\ref{c2(I)})
for the scalar problem. It turns out to be a measure of the sensitivity
of the whole weakly efficient solution set with respect to the tilt
perturbations on the objective function.
\begin{defn}
\noindent Let $f$ be a function in $\mathcal{C}^{1,1}(B(0,r)),$
and $WE_{f}(p)$ be the set of weakly efficient solutions of problem
(VO$_{p}$). The \emph{condition number} ${c}^{*}(f)$ is defined
as follows: \begin{equation}
{c}^{*}(f)=\limsup_{p,q\to0,\, p\neq q}\dfrac{d_{H}\left(WE_{f}(p),WE_{f}(q)\right)}{\left\Vert p-q\right\Vert }.\label{c*WE}\end{equation}

\end{defn}
\noindent The global notion of condition number $c^{*}(f)$ is consistent
with the pointwise approach to conditioning introduced in the former
section through $c(\bar{x},f).$ Indeed, the next proposition holds:
\begin{prop}
\noindent Let $f\in\mathcal{C}^{1,1}(B(0,r)).$ Then, $c(\bar{x},f)<+\infty$
for every $\bar{x}\in WE_{f}(0)$ if and only if $c^{*}(f)<+\infty.$ \end{prop}
\begin{proof}
\noindent From the definition of $c(\bar{x},f),$ for every $\bar{x}\in WE_{f}(0)$
there exists $k_{\bar{x}}>0,$ $\mathcal{V}(\bar{x})$ and $\mathcal{U}_{\bar{x}}(0)$
such that \begin{equation}
WE_{f}(p)\cap\mathcal{V}(\bar{x})\subset WE_{f}(q)+k_{\bar{x}}\|p-q\|B(0,r),\label{inclusione aubin}\end{equation}
 for every $p,q\in\mathcal{U}_{\bar{x}}(0).$ The family of sets $\{\mathcal{V}(\bar{x})\}_{\bar{x}\in WE_{f}(0)}$
is an open covering of the compact set $WE_{f}(0),$ therefore there
exist $\{\bar{x}_{i}\}_{i=1}^{k}$ such that $WE_{f}(0)\subset\cup_{i=1}^{k}\mathcal{V}(\bar{x}_{i}).$
Set $\mathcal{U}(0)=\cap_{i=1}^{k}\mathcal{U}_{\bar{x}_{i}}(0),$
and $\bar{k}=\max_{i}k_{\bar{x}_{i}}.$ We have that, for every $i=1,2,\dots,k,$
and for every $p,q\in\mathcal{U}(0),$ \[
WE_{f}(p)\cap\mathcal{V}(\bar{x}_{i})\subset WE_{f}(q)+k_{\bar{x}_{i}}\|p-q\|B(0,r)\subset WE_{f}(q)+\bar{k}\|p-q\|B(0,r).\]
 Taking the union of the l.h.s. for $i=1,2,\dots,k,$ we get \[
WE_{f}(p)\cap(\cup_{i=1}^{k}\mathcal{V}(\bar{x}_{i}))\subset WE_{f}(q)+\bar{k}\|p-q\|B(0,r).\]
 Since the set $\cup_{i=1}^{k}\mathcal{V}(\bar{x}_{i})$ is a neighborhood
of $WE_{f}(0)$ and the map $p\mapsto WE_{f}(p)$ is upper semicontinuous,
if $p$ is small enough we obtain that \[
WE_{f}(p)\cap(\cup_{i=1}^{k}\mathcal{V}(\bar{x}_{i}))=WE_{f}(p),\]
 thereby showing that, for every $p,q\in\mathcal{U}(0),$ \[
WE_{f}(p)\subset WE_{f}(q)+\bar{k}\|p-q\|B(0,r),\]
 i.e., $d_{H}(WE_{f}(p),WE_{f}(q))\le\bar{k}\|p-q\|.$ The converse
is trivial.
\end{proof}
In case of an $\mathbb{R}_{+}^{m}$--convex function $f$, the set
of weakly efficient solutions $WE_{f}(p)$ can be recovered by the
inverse image of $p$ via $H_{f}$. Consequently, the global condition
number can be defined by

\begin{equation}
c^{*}(f)=\limsup_{p,q\to0,\, p\neq q}\frac{d_{H}(H_{f}^{-1}(p),H_{f}^{-1}(q))}{\|p-q\|}.\label{c*H}\end{equation}

The next proposition shows that, for $\mathbb{R}_{+}^{m}$--convex
functions, the condition number is strictly positive. A similar result
holds in the scalar case (see Lemma 3.2 in \cite{Zol03}).
\begin{prop}
\label{Lemma c positivo} Let $f$ be an $\mathbb{R}_{+}^{m}$--convex
function in $\mathcal{C}^{1,1}(B(0,r)).$ Then $c^{*}(f)>0.$\end{prop}
\begin{proof}
By contradiction, let us suppose that ${c}^{*}(f)=0.$ Now we have
that\begin{align*}
0 & =\limsup_{p,q\rightarrow0,\, p\neq q}\dfrac{d_{H}\left(WE_{f}(p)),WE_{f}(q)\right)}{\left\Vert p-q\right\Vert }\\
 & \geq\limsup_{p\rightarrow0}\dfrac{d_{H}\left(WE_{f}(p),WE_{f}(0)\right)}{\left\Vert p\right\Vert }\\
 & \geq\limsup_{p\rightarrow0}\dfrac{e\left(WE_{f}\left(p\right),WE_{f}(0)\right)}{\left\Vert p\right\Vert };\end{align*}
 hence, for every choice of $x_{p}\in WE_{f}(p),$ we have \begin{equation}
\limsup_{p\rightarrow0}\dfrac{d\left(x_{p},WE_{f}(0)\right)}{\left\Vert p\right\Vert }=0=\lim_{p\rightarrow0}\dfrac{d\left(x_{p},WE_{f}(0)\right)}{\left\Vert p\right\Vert }.\label{limite superiore}\end{equation}
 Since the set $WE_{f}(0)$ is closed and contained in the open set
${\rm int}B(0,r)$, we can always choose a sequence $\left\{ x_{s}\right\} \subset{\rm int}B(0,r)\smallsetminus WE_{f}(0)$
and a point $x_{0}\in WE_{f}(0)$ such that $x_{s}\rightarrow x_{0}$.
Let us consider the sequence $\left\{ v_{f}(x_{s})\right\} \subset\mathbb{R}^{n},$
where $-v_{f}(x_{s})$ is the minimal norm element of $H_{f}(x_{s}).$
We remark that $v_{f}(x_{s})\neq0$ by the convexity of $f.$ By the
continuity of $v_{f}$ (see, for instance, \cite{AubCell84}), we
have that $v_{f}(x_{s})\rightarrow v_{f}(x_{0})=0.$

For every $s\in\mathbb{N}\smallsetminus\left\{ 0\right\} $, set \[
p_{s}=-v_{f}(x_{s}),\]
 and consider the function $f_{p_{s}}:B(0,r)\rightarrow\mathbb{R}^{m}$
defined by\[
f_{p_{s}}(x)=f(x)-[p_{s}]\cdot x.\]
 Hence the minimal norm element of the set\[
H_{f_{p_{s}}}(x_{s})=\left\{ \sum_{i=1}^{m}\lambda_{i}\nabla f_{i}(x_{s})+v_{f}(x_{s}):\sum_{i=1}^{m}\lambda_{i}=1,\,\lambda_{i}\geq0,\, i=1,...,m\right\} \]
 is $0.$ Then, by the convexity assumption, we have that $x_{s}\in WE_{f}(p_{s})$.
Choose a point $x_{s}^{0}\in WE_{f}(0)$ such that \[
d\left(x_{s},WE_{f}(0)\right)=\left\Vert x_{s}-x_{s}^{0}\right\Vert .\]
 Hence, by (\ref{limite superiore}), taking $p=p_{s}=-v_{f}(x_{s}),$
we get \begin{equation}
\lim_{p_{s}\rightarrow0}\dfrac{d\left(x_{s},WE_{f}(0)\right)}{\left\Vert p_{s}\right\Vert }=\lim_{s\rightarrow+\infty}\dfrac{\left\Vert x_{s}-x_{s}^{0}\right\Vert }{\left\Vert v_{f}(x_{s})\right\Vert }=0.\label{eq:Lemma c positivo - 5}\end{equation}
From (\ref{eq:Lemma c positivo - 5}) and taking into account that
$s_{f}(x)=\|v_{f}(x)\|,$ we have \begin{equation}
\dfrac{s_{f}(x_{s})}{\left\Vert x_{s}-x_{s}^{0}\right\Vert }=\dfrac{\left\Vert v_{f}\left(x_{s}\right)\right\Vert }{\left\Vert x_{s}-x_{s}^{0}\right\Vert }\longrightarrow+\infty\;{\rm as\;\quad}s\rightarrow+\infty.\label{eq:nolip}\end{equation}
 Since $v_{f}(x_{s}^{0})=0$ for every $s\in\mathbb{N}\smallsetminus\left\{ 0\right\} $,
relation (\ref{eq:nolip}) contradicts the Lipschitz continuity of
$s_{f}(x)$ proved in Proposition \ref{Lip H}. \end{proof}
\begin{rem}
The proposition above holds under weaker assumptions on $f.$ Indeed,
the proof requires only that $f$ is $\mathbb{R}_{+}^{m}$--convex,
in the class $\mathcal{C}^{1},$ and $s_{f}$ is Lipschitz continuous.
\end{rem}
In the sequel of the section we would like to extend to the global
case the distance theorem already considered in Corollary \ref{cor distance 1}.
Let us first introduce the class of functions giving rise to the well--conditioned
problems:
\begin{defn}
Let $f$ be an $\mathbb{R}_{+}^{m}$--convex function in $\mathcal{C}^{1,1}(B(0,r)).$
Then $f$ is said to belong to the class $W_{1}^{*}$ if $c^{*}(f)<+\infty.$
\end{defn}
Likewise the pointwise approach, strong convexity of the objective
functions entails the finiteness of $c^{*}(f)$. Indeed the following
proposition easily follows by Theorem 5.2 in \cite{LeeKimLeeYen98}.
\begin{prop}
\label{c*finito} Let $f\in\mathcal{C}^{1,1}(B(0,r))$. If $f_{i}$
is strongly convex, for every $i=1,2,\dots,m,$ then $c^{*}(f)<+\infty.$
\end{prop}
Let us consider a function $f\in{W}_{1}^{*},$ and let $g$ be a function
sufficiently {}``close'' to $f,$ according to the distance $d^{*}$
defined in (\ref{d*}). We will prove that under appropriate assumptions,
the perturbed function $g$ will give rise to a well--conditioned
problem. Moreover, the condition number of the perturbed function
will be bounded from above by the reciprocal of the distance between
$f$ and $g$.

As in the previous section we will consider the class $P_{f}$ of
functions defined in (\ref{Pf}) obtained by perturbing each component
of $f$ with the same real--valued function $h$.

Our approach is based on the study of the fixed points of a suitable
set--valued map. The next theorem collects two known results: one
of them concerns the existence of fixed points for contractions with
closed values, and the other one provides an upper bound that can
be established on the distance between the fixed points of two contractions
with closed values. This theorem will play a key role in the proof
of the main result of this section.

Let us denote by $\mathcal{F}(S)$ the set of the fixed points of
the set--valued map $S:X\rightrightarrows X$, i.e.\[
\mathcal{F}(S)=\left\{ x\in X:\, x\in S(x)\right\} .\]

\begin{thm}
\label{teorema Lim}(see Theorem 5 in \cite{Nadler1969} and Lemma
1 in \cite{Lim1985}) Let $X$ be a complete metric space, and let
$S_{1},S_{2}:X\rightrightarrows X$ be two contractions with constant
$\theta$ and closed values. Then $\mathcal{F}(S_{1})$ and $\mathcal{F}(S_{2})$
are nonempty sets; moreover, \[
d_{H}\left(\mathcal{F}(S_{1}),\mathcal{F}(S_{2})\right)\leq\dfrac{1}{1-\theta}\sup_{x\in X}d_{H}\left(S_{1}(x),S_{2}(x)\right).\]

\end{thm}
The following lemma characterizes, for functions $g\in P_{f},$ the
set--valued map $H_{g}^{-1},$ in terms of the fixed points of a suitable
map.
\begin{lem}
\label{Lemma 0 global } Let $f:B(0,r)\to\mathbb{R}^{m}$ and $h:B(0,r)\to\mathbb{R}$
be differentiable functions on $B(0,r)$. Then \[
H_{g}^{-1}(p)=\mathcal{F}\left(H_{f}^{-1}(p-\nabla h(\cdot))\right),\]
 where $g=f+h\mathbf{e}.$ \end{lem}
\begin{proof}
Let $x\in H_{g}^{-1}(p)$. Then\[
\sum_{i=1}^{m}\lambda_{i}\nabla f_{i}(x)+\nabla h(x)=p,\]
 for some nonnegative $\lambda_{i}$, $i=1,...,m$, such that $\sum_{i=1}^{m}\lambda_{1}=1.$
Therefore $p-\nabla h(x)\in H_{f}(x)$ and hence $x\in H_{f}^{-1}(p-\nabla h(x))$,
i.e. $x\in\mathcal{F}\left(H_{f}^{-1}(p-\nabla h(\cdot))\right).$
In a similar way we can prove that $H_{g}^{-1}(p)\supseteq\mathcal{F}\left(H_{f}^{-1}(p-\nabla h(\cdot))\right).$
\end{proof}
We are now in the position to state our main result.
\begin{thm}
Let $f$ be a function in the class $W_{1}^{*},$ and let $g=f+h\mathbf{e}\in P_{f}$
be such that \end{thm}
\begin{itemize}
\item [i)] $g$ is $\mathbb{R}_{+}^{m}$--convex and satisfies $WE_g(0)\subset \mathrm{int}(B(0,r));$
\item [ii)] $\max_{x\in B(0,r)}\left\Vert \nabla h(x)\right\Vert <\delta_{f},$
where $\delta_{f}$ is as in (\ref{eq:WE dentro int});
\item [iii)] $d^{*}(f,g)<\frac{1}{c^{*}(f)}.$
\end{itemize}
Then $g\in W_{1}^{*}.$
\begin{proof}
We divide the proof into a sequence of three steps.

%By assumption ii) there exists a real number $\gamma_{1}>0$ such
%that \[ \left\Vert p-\nabla h(x)\right\Vert <\delta_{f}\]
% for every $x\in B(0,r)$ and for every $p$ such that $\left\Vert p\right\Vert <\gamma_{1}$.
%Therefore, by (\ref{eq:WE dentro int}) and by the convexity assumptions,
%the following inequality holds: \[
%H_{f}^{-1}\left(p-\nabla h(x)\right)=WE_{f}\left(p-\nabla h(x)\right).\]
% Since $WE_{f}\left(p-\nabla h(x)\right)$ is a closed set, the assertion
%is proved.

\emph{Step 1.} We show that the set--valued map $H_{f}^{-1}(p-\nabla
h(\cdot))$ is a contraction of $B(0,r)$ for every small $p.$ First
of all, the upper semicontinuity of $H_f$ entails that
$H_{f}^{-1}(p-\nabla h(\cdot))$ has closed values. Furthermore,
given a positive real number $\eta$, let us consider the quantity
\begin{equation}\label{kappaeta}
K(\eta)=\sup\left\{
\dfrac{d_{H}\left({WE}_{f}(p),{WE}_{f}(q)\right)}{\left\Vert
p-q\right\Vert }:p\neq q,\,\left\Vert p\right\Vert
<\eta,\,\left\Vert q\right\Vert <\eta\right\} .\end{equation} It is
easy to observe that $\lim_{\eta\rightarrow0}K(\eta)=c^{*}(f)$.
Since, by assumption,\[ c^{*}(f)=\limsup_{p,q\rightarrow0,\, p\neq
q}\dfrac{d_{H}\left(WE_{f}(p),WE_{f}(q)\right)}{\left\Vert
p-q\right\Vert }<+\infty,\]
 there exists a real number $\gamma_{1}>0$ such that $K(\eta)<+\infty$
for every $0<\eta<\gamma_{1}$. By assumption ii) there exists a real
number $\gamma_{2}>0$ such that \[ \left\Vert p-\nabla
h(x)\right\Vert <\delta_{f}\]
 for every $x\in B(0,r)$ and for every $p$ such that $\left\Vert p\right\Vert <\gamma_{2}$.
Therefore, by (\ref{eq:WE dentro int}) and by the convexity
assumptions, the following inequality holds: \[
H_{f}^{-1}\left(p-\nabla h(x)\right)=WE_{f}\left(p-\nabla
h(x)\right)\] for every $p$, $\|p\|< \gamma_2.$ Therefore, from the
equality above, we have
\begin{align}
d_{H}\left(H_{f}^{-1}(p-\nabla h(x)),H_{f}^{-1}(p-\nabla
h(y))\right)&\leq  K(\gamma_{3})\left\Vert \nabla h(x)-\nabla
h(y)\right\Vert\nonumber\\
&\leq  K(\gamma_{3})\mathrm{lip}(\nabla h;B(0,r))\left\Vert
x-y\right\Vert\label{eq:Lemma1}
\end{align}
for every $x,y\in B(0,r)$ and for every $p$ such that $\left\Vert
p\right\Vert <\gamma_{3},$ where $\gamma_{3}=\min\left\{
\gamma_{1},\gamma_{2}\right\} $. %As a consequence, for every $x,y\in
%B(0,r)$ and for every $p$ such that $\left\Vert p\right\Vert
%<\gamma_{3}$, we get\begin{equation} d_{H}\left(H_{f}^{-1}(p-\nabla
%h(x)),H_{f}^{-1}(p-\nabla h(y))\right)\leq
%K(\gamma_{3})\mathrm{lip}(\nabla h;B(0,r))\left\Vert x-y\right\Vert
%.\label{eq:Lemma1}\end{equation}

From the equality $d^{*}(f,g)=\mathrm{lip}(\nabla h;B(0,r))$ and the
assumption iii), there exists a real number $\gamma_{4}>0$ such that
for all $\gamma<\gamma_{4}$ we have\[ K(\gamma_{4})d^{*}(f,g)<1.\]
 By (\ref{eq:Lemma1}), we conclude that the set--valued map $H_{f}^{-1}(p-\nabla h(\cdot))$
is a contraction with constant $K(\gamma)d^{*}(f,g)$ for every $p$
such that $\left\Vert p\right\Vert <\gamma,$ where $\gamma=\min\left\{ \gamma_{3},\gamma_{4}\right\} .$

\emph{Step 2.} From (\ref{kappaeta}), and the since
$c^{*}(f)<+\infty$, for every $p,q,$ $p\neq q$, $\|p\|,
\|q\|<\gamma,$ and for every $x\in B(0,r)$,
\begin{equation}
d_{H}\left(H_{f}^{-1}(p-\nabla h(x)),H_{f}^{-1}(q-\nabla
h(x))\right)\leq
K(\gamma)\|p-q\|.\label{eq:teorema1global}\end{equation} From Step 1
and Theorem \ref{teorema Lim}, we have
\[
d_{H}\left(\mathcal{F}\left(H_{f}^{-1}(q_{1}-\nabla
h(\cdot))\right),\mathcal{F}\left(H_{f}^{-1}(q_{2}-\nabla
h(\cdot))\right)\right)\leq\]
\begin{equation}
\leq\dfrac{1}{1-K(\gamma)d^*(f,g)}\sup_{x\in
B(0,r)}d_{H}\left(H_{f}^{-1}(q_{1}-\nabla
h(x)),H_{f}^{-1}(q_{2}-\nabla h(x))\right)\label{teorema2global}
\end{equation}
 for every $q_{1},q_{2}$ such that $\left\Vert q_{1}\right\Vert ,\left\Vert q_{2}\right\Vert <\gamma.$
By (\ref{eq:teorema1global}) and Lemma \ref{Lemma 0 global }, from
the last inequality we deduce that \[
d_{H}\left(H_{g}^{-1}(q_{1}),H_{g}^{-1}(q_{2})\right)\leq\dfrac{K({\gamma})}{1-K(
{\gamma})d^{*}(f,g)}\left\Vert q_{1}-q_{2}\right\Vert \]
 for every $q_{1},q_{2}$ such that $\left\Vert q_{1}\right\Vert ,\left\Vert q_{2}\right\Vert <{\gamma}.$
 Hence, by the definition of condition number, it follows that
 \[
c^{*}(g)\leq\dfrac{c^{*}(f)}{1-c^{*}(f)d^{*}(f,g)}.\]

\end{proof}

\section{Conclusions}

In this work we proposed two approaches for the notion of conditioning
for a multiobjective optimization problem. To our knowledge in the
recent literature on vector optimization there are no results on this
topic. We limit our investigation to the Euclidean setting, and we
consider essentially multiobjective problems involving differentiable
convex functions. The main obstacle one has to overcome when dealing
with this problem is the lack of conditions that fully characterize
the solution set. This explains our choice to restrict our analysis
to the framework of vector--valued convex functions. Moreover, following
our approach we can prove a distance--type theorem only for a special
class of perturbed functions. As a matter of fact, a relevant problem
is the evaluation of the effect of a more general class of perturbations
on the corresponding solution map.

For these reasons we deem that to deal with condition numbers in multiobjective
optimization taking into account general perturbations is not an easy
task and it may require a completely new approach.

\bigskip{}

\textbf{Acknowledgement }The second and the third authors were partially
supported by the Ministerio de Ciencia e Innovación (Spain) under
project MTM2009-09493.

\end{document}